\documentclass[12pt,twoside]{amsart}
\pdfoutput=1

\usepackage[T2A]{fontenc}
\usepackage[utf8]{inputenc}

\usepackage[a4paper]{geometry}
\usepackage{amssymb}
\usepackage{amsmath}

\usepackage[sort&compress,numbers]{natbib}
\newcommand*{\doi}[1]{\href{https://doi.org/\detokenize{#1}}{doi: \detokenize{#1}}} 

\usepackage[dvipsnames,hyperref]{xcolor}

\usepackage[\GinDriver,unicode]{hyperref}

\hypersetup{
  setpagesize  = false,
  breaklinks   = true,
  pdfborder    = {0 0 0},
  colorlinks   = true, 
  urlcolor     = magenta, 
  linkcolor    = blue, 
  citecolor    = Green,  
  filecolor    = cyan   
}

\newcommand{\RR}{\mathbb R}

\newcommand{\sech}{\mathop{\mathrm{sech}}\nolimits}

\newcommand{\ee}{\mathrm{e}}

\newcommand{\manifold}[1]{\mathcal{#1}}
\newcommand{\M}{\manifold{M}}
\newcommand{\D}{\manifold{D}}

\newcommand{\vect}[1]{\mathrm{#1}} 
\newcommand{\x}{\vect{x}}

\newcommand{\va}{\vect{a}}
\newcommand{\vb}{\vect{b}}

\newcommand{\vX}{\vect{X}}
\newcommand{\vY}{\vect{Y}}

\newcommand{\vl}{l}

\newtheorem{thm}{Theorem}[section]

\theoremstyle{remark}
\newtheorem{remark}{Remark}[section]
\newtheorem{example}{Example}[section]
\theoremstyle{definition}
\newtheorem{dfn}{Definition}[section]

\makeatletter
\@addtoreset{equation}{section}
\makeatother

\newcommand{\ds}{\displaystyle}

\begin{document}

\title[Canonical Coordinates and Natural Equation for Lorentz Surfaces in $\mathbb R^3_1$]
{Canonical Coordinates and Natural Equation for Lorentz Surfaces in $\mathbb R^3_1$}

\author{Krasimir Kanchev}
\author{Ognian Kassabov}
\author{Velichka Milousheva}

\address {Department of Mathematics and Informatics, Todor Kableshkov University of Transport,
158 Geo Milev Str., 1574, Sofia, Bulgaria}%
\email{kbkanchev@yahoo.com}%

\address{Institute of Mathematics and Informatics, Bulgarian Academy of Sciences,
Acad. G. Bonchev Str. bl. 8, 1113, Sofia, Bulgaria}
\email{okassabov@math.bas.bg}

\address{Institute of Mathematics and Informatics, Bulgarian Academy of Sciences,
Acad. G. Bonchev Str. bl. 8, 1113, Sofia, Bulgaria}
\email{vmil@math.bas.bg}


\begin{abstract}
We consider Lorentz surfaces in $\mathbb R^3_1$ satisfying the condition $H^2-K\neq 0$,
where $K$ and $H$ are the Gauss curvature and the mean curvature, respectively, and call them Lorentz surfaces of general type.
For this class of surfaces we introduce special isotropic coordinates, which we call canonical, and show that the coefficient $F$ of the first fundamental form   and the mean curvature $H$, expressed in terms of the canonical coordinates, satisfy a special integro-differential equation which we call a natural equation of the Lorentz surfaces of general type. Using this natural equation we prove a fundamental theorem of Bonnet type for Lorentz surfaces of general type. We consider the special cases of  Lorentz surfaces of constant non-zero mean curvature and  minimal Lorentz surfaces.
Finally, we give examples of Lorentz surfaces illustrating the developed theory.

\end{abstract}

\maketitle

\tableofcontents


\section{Introduction} \label{S:Intro}

The question of describing surfaces with prescribed mean or Gauss curvature in the Euclidean 3-space $\RR^3$ and also in the other Riemannian space forms have been subject of an intensive study. Especially, the geometry of spacelike or timelike surfaces in the Minkowski 3-space $\RR^3_1$ has been  of wide interest.  For example, a Kenmotsu-type representation formula for spacelike surfaces with prescribed mean curvature was obtained by K. Akutagawa  and S. Nishikawa in \cite{Aku-Nish}. In \cite{Gal-Mart-Mil}, G\'alvez et al. 
obtained a representation for spacelike surfaces in $\RR^3_1$  using the Gauss map and the conformal structure given by the second fundamental form. M. A. Magid proved that the Gauss map and the mean curvature of a timelike surface 
satisfy a system of partial differential equations and  found a Weierstrass representation formula for timelike surfaces in $\RR^3_1$ \cite{Mag}. 
Timelike surfaces in $\RR^3_1$ with prescribed Gauss curvature and Gauss map are studied in \cite{Al-Esp-Gal} where  
a Kenmotsu-type representation for such surfaces is given. This representation is used to  classify the complete
timelike surfaces with positive constant Gaussian curvature in terms of
harmonic diffeomorphisms between simply connected Lorentz surfaces and the universal covering of
the de Sitter Space.

On the other hand, it is known that the minimal Lorentz surfaces in $\RR^3_1$, $\RR^4_1$, and $\RR^4_2$ can be parametrized by special isothermal coordinates, called \textit{canonical},  such that the main invariants (the Gauss curvature and the normal curvature) of the surface satisfy a system of partial differential equations called a \textit{system of natural PDEs}. The geometry of the corresponding minimal surface is determined  
by the solution of this  system of natural PDEs. 

In \cite{G-M-2013-1},  canonical coordinates for the class of minimal Lorentz surfaces in the Minkowski space $\RR^4_1$  are introduced and the following 
system of natural PDEs is obtained:
\begin{equation}\label{Nat_Eq_K_kappa_R41-tl}
\begin{array}{lll}
\sqrt[4]{K^2+\varkappa^2\phantom{\big|}}\; \Delta^h \ln \sqrt[4]{K^2+\varkappa^2\phantom{\big|}}  &=& 2K; \\[1.5ex]
\sqrt[4]{K^2+\varkappa^2\phantom{\big|}}\; \Delta^h \arctan \ds\frac{\varkappa}{K} &=& 2\varkappa;
\end{array}  \qquad\quad K^2+\varkappa^2\neq 0,
\end{equation} 
where $K$ is the Gauss curvature,  $\varkappa$ is the curvature of the normal connection (the normal curvature), and $\Delta^h$ is the  hyperbolic Laplace operator in $\RR^2_1$.

Similar results are obtained for minimal Lorentz surfaces in the pseudo-Euclidean space with neutral metric 
 $\RR^4_2$ in \cite{A-M-1} and \cite{Kanchev2020}. 
The corresponding system of PDEs has the following form:
\begin{equation}\label{Nat_Eq_K_kappa_R42-tl}
\begin{array}{lll}
\sqrt[4]{\big|K^2-\varkappa^2\big|}\; \Delta^h\ln \sqrt[4]{\big|K^2-\varkappa^2\big|} &=& 2K\,; \\[1.5ex]
\sqrt[4]{\big|K^2-\varkappa^2\big|}\; \Delta^h\ln \left|\ds\frac{\vphantom{\mu^2}K+\varkappa}{K-\varkappa}\right|&=& 4\varkappa\;;
\end{array}  \qquad\quad K^2-\varkappa^2\neq 0.
\end{equation} 

 The minimal Lorentz surfaces in  $\RR^3_1$ can also be considered as surfaces in  $\RR^4_1$ or $\RR^4_2$, in which cases $\varkappa=0$.
So, systems \eqref{Nat_Eq_K_kappa_R41-tl} and \eqref{Nat_Eq_K_kappa_R42-tl} are reduced to one PDE:
\begin{equation}\label{Nat_Eq_K_R31-tl}
\sqrt{|K|}\: \Delta^h \ln \sqrt{|K|} = 2K; \qquad K \neq 0,
\end{equation}
which is the natural equation of minimal Lorentz surfaces in $\RR^3_1$. Of course, the results in this case can be directly obtained.
In \cite{Ganchev2008-1},  canonical coordinates are introduced for minimal Lorentz surfaces in $\RR^3_1$ and equations equivalent to \eqref{Nat_Eq_K_R31-tl} are derived.

Thus the following natural question arises: \emph{How to generalize the concepts of canonical coordinates and natural equation
for a wider class of Lorentz surfaces in  $\RR^3_1$ than that of the minimal ones?} 
The class of Weingarten Lorentz surfaces in $\RR^3_1$ with different real principal curvatures (which is equivalent to $H^2-K > 0$, 
where $K$ and $H$ are the Gauss curvature and the mean curvature, respectively) is considered in \cite{Ganchev-Mihova-2013-1}.  
Canonical principal coordinates are introduced for this class of surfaces and a natural  non-linear partial differential equation is derived, 
which is equivalent to \eqref{Nat_Eq_K_R31-tl} in the case of a minimal surface. 

In the present paper, we propose an alternative approach. We consider Lorentz surfaces in  $\RR^3_1$ satisfying  $H^2-K\neq 0$ and call them \textsl{surfaces of general type}. We introduce special isotropic coordinates (which we call \textit{canonical}) for these surfaces  and obtain a \emph{natural integro-differential equation}. The natural equation for the class of minimal Lorentz surfaces is given by 
\begin{equation*}
\sqrt{|K|}\big(\ln \sqrt{|K|}\,\big)_{uv} = K; \qquad K \neq 0.
\end{equation*}
It can be reduced to \eqref{Nat_Eq_K_R31-tl} by changing the isotropic coordinates with isothermal ones.  This shows,
that the newly obtained results for an arbitrary Lorentz surface of general type in $\RR^3_1$ generalize the known results for the case of a minimal Lorentz surface.

In Section \ref{S:Prelim}, we give some basic formulas for Lorentz surfaces in  $\RR^3_1$ parametrized by arbitrary isotropic coordinates. We present the
 Gauss and Codazzi equations  in terms of these coordinates and formulate the fundamental theorem of Bonnet type.

In Section \ref{Can_Coord-LorSurf_R31}, we introduce the notion of \textit{canonical isotropic coordinates} for the class of Lorentz surfaces of general type ($H^2-K\neq 0$) in $\RR^3_1$.
We prove existence and uniqueness theorems for these coordinates and give the relation between the canonical coordinates and the natural parameters of the isotropic curves of the surface. 

 In Section \ref{Nat_Eq-LorSurf_R31}, we consider Lorentz surfaces of general type parametrized by canonical coordinates and show that the coefficient $F$ of the first fundamental form and the mean curvature  $H$ of such surface satisfy the following integro-differential equation
\begin{equation*}
\frac{F F_{uv} -  F_u F_v}{F} = \textstyle
\left(\varepsilon_1 + \int_{v_0}^v\, F(u,s)H_u(u,s)\,ds\right) 
\left(\varepsilon_2 + \int_{u_0}^u\, F(s,v)H_v(s,v)\,ds \right) - F^2H^2,
\end{equation*}
$\varepsilon_1 = \pm 1, \varepsilon_2 = \pm 1$.
We call it the \textit{natural equation} of the Lorentz surfaces in  $\RR^3_1$ and prove a fundamental theorem of  Bonnet type.
We consider in detail the special cases of a Lorentz surface with  non-zero constant mean curvature and  a minimal Lorentz surface.

 In Section \ref{S:Examples_Can_Coord-LorSurf_R31}, we give examples of different types of Lorentz surfaces and their canonical coordinates in 
$\RR^3_1$.


\section{Preliminaries} \label{S:Prelim}

Let  $\RR^3_1$  be the standard three-dimensional pseudo-Euclidean space in which the indefinite inner scalar product is given by the formula:
\begin{equation*}
\langle \va ,\vb \rangle = -a_1b_1+a_2b_2+a_3b_3.
\end{equation*}
Let $\M=(\D ,\x)$ be a Lorentz surface in $\RR^3_1$, where $\D\subset\RR^2$ and $\x : \D \to \RR^3_1$ is an immersion.
The coefficients of the first fundamental form of $\M$ are denoted as usually by $E, F, G$  and  $L, M, N$ denote the coefficients of the second fundamental form.
Then, the Gauss curvature $K$ and the mean  curvature   $H$ of $\M$ are given by the formulas (see \cite{Anciaux-1,Lopez2014}):
\begin{equation*}
K = \frac{LN-M^2}{EG-F^2}; \qquad H = \frac{EN-2FM+GL}{2(EG-F^2)}.
\end{equation*}

 In a neighbourhood of each point of  $\M$ there exist isotropic coordinates  $(u,v)$ such that $E=G=0$ \cite{Anciaux-1}. Such parameters  are also called null coordinates  \cite{Inog-1}.  It can easily be seen that if $(u,v)$ and $(\tilde u,\tilde v)$ are two different pairs   of isotropic coordinates in a  neighbourhood of a fixed point, then they are related either by  $u = u(\tilde u)$,  $v = v(\tilde v)$, or  $u = u(\tilde v)$,  $v = v(\tilde u)$.

Further, we consider a surface $\M$ parametrized by isotropic coordinates and without loss of generality we assume that $F>0$. Then, the formulas for $K$ and $H$ take the following form: 
\begin{equation}\label{H_K-Null_R^3_1-tl}
K = \frac{M^2-LN}{F^2}; \qquad H = \frac{M}{F}.
\end{equation}

Consider  the tangent vector fields $\vX = \x_u, \; \vY = \x_v$ and denote by 
 $\vl$ the unit normal vector field $\vl = \displaystyle{\frac{\x_u \times \x_v}{|\x_u\times \x_v|}}$ such that  
$\{\vX,\vY,\vl\}$ be a positively oriented  frame field in  $\RR^3_1$.
Since $\M$ is parametrized by isotropic coordinates, we have:
\begin{equation*}
\vX^2=\vY^2=0; \quad \vl^2=1; \quad \langle \vX ,\vY \rangle = F; \quad \langle \vX ,\vl \rangle = \langle \vY ,\vl \rangle = 0.
\end{equation*}
Hence, we get 
\begin{equation*}
\vX_v=\vY_u; \quad 
\langle \vX_u ,\vX \rangle = \langle \vX_v ,\vX \rangle = \langle \vY_u ,\vY \rangle = \langle \vY_v ,\vY \rangle = 
\langle \vl_u ,\vl \rangle = \langle \vl_v ,\vl \rangle = 0.
\end{equation*}

Using the last equalities we obtain the following Frenet-type formulas for the frame field $\{\vX,\vY,\vl\}$:
\begin{equation}\label{Frene_X_Y_n-R^3_1-tl}
\left|
\begin{array}{llrrr}
        \vX_u &=&  \ds\frac{F_u}{F} \vX  &                                  & + \;\;\,    L \vl; \\[1.7ex]
        \vY_u &=&                        &                                  &   \  \;\;\, M \vl; \\[1.7ex]
				\vl_u  &=& -\ds\frac{M}{F}   \vX  & -\ \ \ \ds\frac{L}{F}   \vY;     &                
\end{array}\right. \quad
\left|
\begin{array}{llrrr}
        \vX_v &=&                        &                                  &     \;\;\,   M \vl; \\[0.7ex]
        \vY_v &=&                        &  \ds\frac{F_v}{F}       \vY\;    & +\  \;\;\,   N \vl; \\[1.7ex]
				\vl_v  &=& -\ds\frac{N}{F}   \vX  & -\ \ \ \ds\frac{M}{F}   \vY.     &                 
\end{array}\right. 
\end{equation}

\smallskip
The integrability conditions of \eqref{Frene_X_Y_n-R^3_1-tl}, considered as a system of PDE for the triple $(\vX,\vY,\vl)$, imply the following Gauss equation
 \begin{equation}\label{Gauss_Null_2-R^3_1-tl} 
\frac{F F_{uv} -  F_u F_v}{F} = L N - M^2
\end{equation}
and the Codazzi equations
\begin{equation}\label{Codazzi_12_Null_2-R^3_1-tl} 
L_v = F \left(\ds\frac{M}{F}\right)_u; \qquad N_u = F \left(\ds\frac{M}{F}\right)_v.
\end{equation}
Note that \eqref{Gauss_Null_2-R^3_1-tl} and \eqref{H_K-Null_R^3_1-tl} imply 
$K=-\ds\frac{1}{F}(\ln F)_{uv}$, which is the Gauss's  \textit{Theorema Egregium} in the case of isotropic coordinates.
\smallskip

As it is well known, the Gauss and Codazzi  equations are not only necessary, but  also sufficient conditions for the 
existence of a solution to the PDE system  \eqref{Frene_X_Y_n-R^3_1-tl}. This gives us a fundamental Bonnet-type theorem for Lorentz surfaces in $\RR^3_1$.
The proof is analogous to that of the classical theorem for surface in $\RR^3$ (see (\cite{doCarmo}).

\begin{thm}\label{Thm-Bonnet_FLMN_R31-tl}
Let $\M$ be a Lorentz surface in $\RR^3_1$ parametrized by isotropic coordinates.
Then, the coefficients $F$, $L$, $M$, $N$ of the first and the second fundamental form of $\M$ give a solution to the Gauss and Codazzi  equations \eqref{Gauss_Null_2-R^3_1-tl} and \eqref{Codazzi_12_Null_2-R^3_1-tl}.
If $\hat\M$ is obtained from $\M$ by a  proper motion in $\RR^3_1$, then $\hat\M$ generates the same solution to  \eqref{Gauss_Null_2-R^3_1-tl} and \eqref{Codazzi_12_Null_2-R^3_1-tl}. 

 Conversely, if the functions  $F$, $L$, $M$, $N$ satisfy equations \eqref{Gauss_Null_2-R^3_1-tl} and \eqref{Codazzi_12_Null_2-R^3_1-tl}, then, at least locally, there exists  a unique (up to a proper motion in $\RR^3_1$) Lorentz surface $\M$ parametrized by isotropic coordinates, such that the given functions are the coefficients of the first and the second fundamental form of $\M$.
\end{thm}


\section{Canonical isotropic coordinates of  Lorentz surfaces in $\mathbb R^3_1$}\label{Can_Coord-LorSurf_R31}

 In the present section, we will show that the Lorentz surfaces satisfying  $H^2-K\neq 0$  admit special isotropic coordinates which we will call canonical. 
It follows from the Codazzi  equations \eqref{Codazzi_12_Null_2-R^3_1-tl} and the second equality of \eqref{H_K-Null_R^3_1-tl} that 
\begin{equation}\label{new-eq}
L_v = F H_u; \quad N_u = F H_v.
\end{equation}
Integrating the last equalities, we obtain
\begin{equation}\label{LN-FH-R^3_1-tl} 
L = L(u,v_0) + \int_{v_0}^v\, F(u,s)H_u(u,s)\,ds;
	\quad
N = N(u_0,v) + \int_{u_0}^u\, F(s,v)H_v(s,v)\,ds.
\end{equation}

Now, we will try to choose the isotropic coordinates in such a way that $L(u,v_0)$ and $N(u_0,v)$ 
to have the simplest form.

First, let's find the transformation formulas for the coefficients of the first and the second fundamental  form under changes of the isotropic coordinates.
 Consider the following change  of the isotropic coordinates (it preserves the numeration): 
\begin{equation*}\label{u_tild_u-R^3_1-tl}
u = u(\tilde u); \qquad v = v(\tilde v).
\end{equation*}
Then, we have $\tilde F = Fu'v'$, which implies $u'v'>0$, since we have assumed at the beginning that $\tilde F > 0$ and $F > 0$.
In this case, the orientation of the surface does not change, i.e. $\tilde \vl = \vl$.
Then, for the coefficients $\tilde F$, $\tilde L$, $\tilde M$,  $\tilde N$ we have
\begin{equation}\label{FLMN_tild_1-R^3_1-tl}
\tilde F = Fu'v'; \qquad \tilde L = L{u'}^2; \qquad \tilde M = Mu'v'; \qquad \tilde N = N{v'}^2.
\end{equation}

In the case of changing the numeration,  it is sufficient  to consider only the change of coordinates:
\begin{equation*}\label{u_tild_v-R^3_1-tl}
u = \tilde v; \qquad v = \tilde u,
\end{equation*}
since the general case is reduced to this and the previous one. In this case, the orientation of the surface changes, i.e. $\tilde \vl = -\vl$ and we have 
\begin{equation}\label{FLMN_tild_2-R^3_1-tl}
\tilde F = F; \qquad \tilde L = -N; \qquad \tilde M = -M; \qquad \tilde N = -L.
\end{equation}

 The transformation formulas \eqref{FLMN_tild_1-R^3_1-tl} and \eqref{FLMN_tild_2-R^3_1-tl} show that, 
if $L=0$ or $N=0$ for some isotropic coordinates, then $\tilde L=0$ or $\tilde N=0$ for any isotropic coordinates.
Further, we consider surfaces satisfying  $L\neq 0$ and  $N\neq 0$ at least locally. 
It follows from  \eqref{H_K-Null_R^3_1-tl} that
\begin{equation}\label{H2K-Null_R^3_1-tl}
H^2-K = \frac{LN}{F^2},
\end{equation}
which implies that the conditions  $L\neq 0$ and $N\neq 0$ are equivalent to  $H^2-K\neq 0$.

We give the following definition.

\begin{dfn}\label{Def-Gen_Typ_R^3_1-tl}
A Lorentz surface $\M$ in $\RR^3_1$ is said to be of \textit{general type}, if  $H^2-K\neq 0$.
\end{dfn}

The Lorentz surfaces of general type are naturally divided into two subclasses.

\begin{dfn}\label{Def-Kind12_R^3_1-tl}
A Lorentz surface of general type in $\RR^3_1$ is said to be of \textit{first kind} (resp. \textit{second kind}), if  $H^2-K > 0$ (resp  $H^2-K < 0$).
\end{dfn}

\begin{remark}\label{Rem-Gen_Typ_Kind12_R^3_1-tl}
Since  $K$ and $H$ are invariants of $\M$, the property of a surface to be of general type, as well as the kind of the surface, are geometric -- they do not depend on the local parametrization and are invariant under motions in  $\RR^3_1$.
It is known that the discriminant of the characteristic polynomial of the Weingarten map   has the form $D=4(H^2-K)$ \cite{Lopez2014}.
Hence, the surfaces of general type are those surfaces for which the Weingarten map has two different eigenvalues. Moreover, the surfaces of first kind are those with real eigenvalues, the surfaces of second kind are those with complex eigenvalues.
\end{remark}

Now, we will  introduce special isotropic coordinates by the following

\begin{dfn}\label{Def_Can_R31-tl}
Let $\M=(\D ,\x)$ be a Lorentz surface of general type in  $\RR^3_1$ parametrized by isotropic coordinates $(u,v)$, such that $F>0$, and $\x_0=\x(u_0,v_0)$ be a point of $\M$.
We call $(u,v)$  \textit{canonical coordinates with initial point  $\x_0$}, if the coefficients of the second fundamental form satisfy the conditions
\begin{equation}\label{Can_LN_R42-tl}
L(u,v_0) = \varepsilon_1; \qquad N(u_0,v) = \varepsilon_2,
\end{equation}
where $\varepsilon_1=\pm 1$ and $\varepsilon_2=\pm 1$.
\end{dfn}

\begin{remark}\label{Rem_Can_R31-tl}
In the general case,  condition \eqref{Can_LN_R42-tl} for  the coordinates to be canonical depends on the choice of the initial point 
$(u_0,v_0)$.
If $(u_1,v_1)$ is another point in the same neighbourhood, from  \eqref{LN-FH-R^3_1-tl} it is obvious that  
$L(u,v_0)\neq L(u,v_1)$ and $N(u_0,v)\neq N(u_1,v)$, in general. In the special case of a surface with constant mean curvature 
$H$, $L$ is a function of  $u$ and $N$ is a function of $v$, because of \eqref{new-eq}. Hence, for the class of surfaces with constant mean curvature the canonicity of the coordinates does not depend on the choice of the initial point $(u_0,v_0)$. 
\end{remark}

\begin{thm}\label{Can_Coord-exist_R42-tl}
If $\M=(\D ,\x)$ is a Lorentz surface of general type in  $\RR^3_1$ and $\x_0$ is a fixed point, then we can introduce canonical coordinates with initial point $\x_0$.
\end{thm}

\begin{proof}
Let $\x_0$ be a fixed point of $\M$ and $(u,v)$ be isotropic coordinates in a  neighbourhood of $\x_0$ such that  $F>0$ and $\x_0=\x(u_0,v_0)$. 
Consider the change of the coordinates  $u = u(\tilde u)$ and $v = v(\tilde v)$. According to \eqref{FLMN_tild_1-R^3_1-tl} and Definition \ref{Def_Can_R31-tl}, the new coordinates $(\tilde u, \tilde v)$ 
are canonical if and only if
\begin{equation}\label{cond_can_R31-tl}
L(u,v_0){u'}^2 = \tilde L(\tilde u, \tilde v_0) = \pm 1; \qquad N(u_0,v){v'}^2 = \tilde N(\tilde u_0, \tilde v) = \pm 1.
\end{equation}
The signs in the right-hand sides of \eqref{cond_can_R31-tl} are chosen to coincide with the signs of $L$ and $N$, respectively.
Thus we obtain ordinary differential equations  for $u(\tilde u)$ and $v(\tilde v)$ whose solutions have the following form:
\begin{equation}\label{eq_can-sol_R31-tl}
\tilde u = \tilde u_0 + \int_{u_0}^u \sqrt{|L(s,v_0)|}\,ds; \qquad \tilde v = \tilde v_0 + \int_{v_0}^v \sqrt{|N(u_0,s)|}\,ds,
\end{equation}
where $(\tilde u_0,\tilde v_0)$ is arbitrary chosen.
It follows from  \eqref{H2K-Null_R^3_1-tl} that $L\neq 0$ and $N\neq 0$, which imply $\tilde u'>0$ and $\tilde v'>0$.
Hence,  equalities  \eqref{eq_can-sol_R31-tl} define new isotropic coordinates  $(\tilde u,\tilde v)$ satisfying the condition
$\tilde F>0$. Since \eqref{eq_can-sol_R31-tl} is equivalent to \eqref{cond_can_R31-tl}, then 
 $\tilde L(\tilde u, \tilde v_0) = \pm 1$ and $\tilde N(\tilde u_0, \tilde v) = \pm 1$.
So, $(\tilde u,\tilde v)$ are canonical coordinates of $\M$ with initial point $\x_0$.   
\end{proof}

\begin{remark}\label{Rem-Can_Coord-exist_R42-tl}
If $\M$ is a surface of first kind according to Definition \ref{Def-Kind12_R^3_1-tl}, then \eqref{H2K-Null_R^3_1-tl} implies $LN>0$.
It follows from  \eqref{FLMN_tild_2-R^3_1-tl} that a change in the coordinates numeration leads to a change in the signs of  
 $L$ and $N$. Hence, surfaces of first kind admit both canonical coordinates for which $L=N=1$ and canonical coordinates for which
$L=N=-1$.

If $\M$ is of second kind, then  $LN<0$, and hence,  \eqref{FLMN_tild_1-R^3_1-tl} and \eqref{FLMN_tild_2-R^3_1-tl} show that the signs of 
 $L$ and $N$ do not change under changes of the isotropic coordinates. Therefore, the surfaces of second kind can be divided into two subclasses: surfaces with canonical coordinates  such that $L=1$ and $N=-1$, and surfaces with  canonical coordinates such that  
$L=-1$ and $N=1$. 
\end{remark}

Now, we will discuss the question of  uniqueness of the canonical coordinates.

\begin{thm}\label{Can_Coord-uniq_R31-tl}
Let $\M$ be a Lorentz surface of general type in $\RR^3_1$ and $(u,v)$ be canonical coordinates with initial point $x_0 = x(u_0,v_0)$ of $\M$.
Then $(\tilde u, \tilde v)$ is another pair of canonical coordinates with the same initial point $x_0$ if and only if
\begin{equation}\label{uniq_can_coord_R31-tl}
\begin{array}{llr} u \!\!\!&=&\!\!\! \delta\tilde u + c_1;\\ v \!\!\!&=&\!\!\! \delta\tilde v + c_2; \end{array} 
\quad \text{or} \quad 
\begin{array}{llr} u \!\!\!&=&\!\!\! \delta\tilde v + c_1;\\ v \!\!\!&=&\!\!\! \delta\tilde u + c_2, \end{array}
\end{equation}
where $\delta=\pm 1$, $c_1$ and $c_2$ are constants. 
\end{thm}

\begin{proof}
First, we consider the case   $u = u(\tilde u)$ and $v = v(\tilde v)$. 
Equalities \eqref{FLMN_tild_1-R^3_1-tl} imply that  $(\tilde u,\tilde v)$ are canonical coordinates if and only if ${u'}^2=1$, ${v'}^2=1$, and $u'v'>0$. 
The last conditions are equivalent to the first pair of equalities in \eqref{uniq_can_coord_R31-tl}

 The case  $u = u(\tilde v)$ and $v = v(\tilde u)$ reduces to the previous one by means of  \eqref{FLMN_tild_2-R^3_1-tl}.
\end{proof}

The meaning of the last theorem is that the canonical coordinates are uniquely determined up to a numeration, a sign and an additive constant.

\medskip

 At the end of this section we will characterize the canonical coordinates in terms of the null curves lying on the considered surfaces. 
Recall that, if  $\alpha$ is a null curve  (${\alpha'}^2=0$) in $\RR^3_1$, then ${\alpha''}^2\ge 0$.
The null curves satisfying ${\alpha''}^2>0$ are called  \emph{non-degenerate}. It is known that these curves admit a parametrization such that 
${\alpha''}^2=1$ \cite{Lopez2014}. Such parameter is known in the literature as a \textit{natural parameter} or 
\textit{pseudo arc-length parameter}, since it plays a role similar to the role
of the arc-length parameter for non-null curves \cite{Duggal-Jin}.

\begin{thm}\label{Can_Coord-Nat-Param_R31-tl}
Let $\M=(\D ,\x)$ be a Lorentz surface  in  $\RR^3_1$ parametrized by isotropic coordinates $(u,v)$, such that $F>0$, and $\x_0=\x(u_0,v_0)$ be a point of $\M$.
Then, $\M$ is of general type if and only if the null curves lying on $\M$  are non-degenerate.  
The coordinates  $(u,v)$ of $\M$ are canonical with initial point $\x_0$ if and only if they are  natural parameters of the null curves on  $\M$ passing through $\x_0$.
\end{thm}

\begin{proof}
The Frenet formulas  \eqref{Frene_X_Y_n-R^3_1-tl} imply $\x_{uu}^2=\vX_u^2=L^2$ and $\x_{vv}^2=\vY_v^2=N^2$. Hence, the 
$u$-lines and  $v$-lines are non-degenerate if and only if  $L\neq 0$ and $N\neq 0$, which is equivalent to $\M$ being of general type. 
Moreover, $u$ is a natural parameter  ($\x_{uu}^2=1$) of the null curve $\x(u,v_0)$ passing through $\x_0$ if and only if $L^2(u,v_0)=1$.
Analogously,  $v$ is a natural parameter  ($\x_{vv}^2=1$) of the null curve  $\x(u_0,v)$  passing through $\x_0$ if and only if $N^2(u_0,v)=1$. 
The last conditions are equivalent to  $L(u,v_0)=\pm 1$ and $N(u_0,v)=\pm 1$, which means that the coordinates $(u,v)$ of $\M$ are canonical with initial point $\x_0$.
\end{proof}


\section{Natural equation of Lorentz surfaces of general type in $\mathbb R^3_1$}\label{Nat_Eq-LorSurf_R31}

 In this section, we will consider the Gauss and Codazzi equations of a Lorentz surface of general type  $\M=(\D ,\x)$ in $\RR^3_1$
parametrized by canonical isotropic coordinates  $(u,v)$ with initial point  $\x_0=\x(u_0,v_0) \in \M$. 
In such case, the coefficients of the second fundamental form are expressed by the coefficient  $F$ of the first fundamental form, the mean curvature $H$, and the constants  $\varepsilon_1$ and $\varepsilon_2$ (see Definition \ref{Def_Can_R31-tl}).
It follows from \eqref{H_K-Null_R^3_1-tl}, \eqref{LN-FH-R^3_1-tl}, and \eqref{Can_LN_R42-tl} that:
\begin{equation}\label{LMN-FH-R^3_1-tl} \textstyle
L = \varepsilon_1 + \int_{v_0}^v\, F(u,s)H_u(u,s)\,ds;
	\quad M = FH; \quad
N = \varepsilon_2 + \int_{u_0}^u\, F(s,v)H_v(s,v)\,ds,
\end{equation}
where $\varepsilon_1=\pm 1$;\: $\varepsilon_2=\pm 1$,  the signs depending on the kind of the surface.
Substituting these expressions in the Gauss equation \eqref{Gauss_Null_2-R^3_1-tl}, we obtain:
\begin{equation}\label{Nat_Eq_R^3_1-tl} 
\frac{F F_{uv} -  F_u F_v}{F} = \textstyle
\left(\varepsilon_1 + \int_{v_0}^v\, F(u,s)H_u(u,s)\,ds\right) 
\left(\varepsilon_2 + \int_{u_0}^u\, F(s,v)H_v(s,v)\,ds \right) - F^2H^2.
\end{equation}
Consequently, $F$ and $H$ give a  solution to the integro-differential equation \eqref{Nat_Eq_R^3_1-tl}, 
which we  call  the \emph{\bfseries natural equation} of the Lorentz surfaces of general type in $\RR^3_1$.
The converse is also true. Namely, the following Bonnet-type theorem holds:

\begin{thm}\label{Thm-Bonnet_Nat_Eq_FH_R31-tl}
Let $\M=(\D ,\x)$ be a Lorentz surface of general type in $\RR^3_1$ and $(u,v)$ be canonical isotropic coordinates with initial point
 $\x_0=\x(u_0,v_0) \in \M$. Then, the coefficient $F$ of the first fundamental form and the mean curvature  $H$ of $\M$ give a  solution to the natural equation  \eqref{Nat_Eq_R^3_1-tl}.
If  $\hat\M$ is obtained from  $\M$ by a proper motion in  $\RR^3_1$, then $\hat\M$ generates the same solution to 
 \eqref{Nat_Eq_R^3_1-tl}. 

 Conversely, let  $F>0$ and $H$ be functions of $(u,v)$ defined in a neighbourhood of $(u_0,v_0)$ and satisfying the natural equation 
\eqref{Nat_Eq_R^3_1-tl}, where $\varepsilon_1=\pm 1$ and $\varepsilon_2=\pm 1$. Then, at least locally, there exists a unique (up to a proper motion in $\RR^3_1$) Lorentz surface of general type in $\RR^3_1$ defined by $x = x(u,v)$ in canonical isotropic coordinates with initial point $x_0 = x(u_0, v_0)$, such that the given functions $F$ and $H$ are the non-zero coefficient of the first fundamental form and the mean curvature, respectively, and the signs of the corresponding coefficients $L$ and $N$ of the second fundamental form coincide with the signs of $\varepsilon_1$ and $\varepsilon_2$.
\end{thm}

\begin{proof}
We have already seen that the coefficient $F$  and the mean  curvature $H$ of a surface with the given properties satisfy the natural equation.
Now we will prove the converse.

Given the functions $F$ and $H$ satisfying \eqref{Nat_Eq_R^3_1-tl}, we define functions 
 $L$, $M$, and $N$ by equalities \eqref{LMN-FH-R^3_1-tl}. Then,  \eqref{Nat_Eq_R^3_1-tl} implies that the quadruple  $F$, $L$, $M$, $N$ satisfies the Gauss equation \eqref{Gauss_Null_2-R^3_1-tl}.
Differentiating  \eqref{LMN-FH-R^3_1-tl} we get that the Codazzi equations \eqref{Codazzi_12_Null_2-R^3_1-tl} are also fulfilled.
Applying Theorem \ref{Thm-Bonnet_FLMN_R31-tl} we get a Lorentz surface $\M$ parametrized by isotopic coordinates whose coefficient of the first fundamental form is the given function 
 $F$ and the coefficients of the second fundamental form are the functions  $L$, $M$, $N$, defined by \eqref{LMN-FH-R^3_1-tl}.
Comparing \eqref{LN-FH-R^3_1-tl} with \eqref{LMN-FH-R^3_1-tl} we obtain
 $L(u,v_0) = \varepsilon_1$,  $N(u_0,v) = \varepsilon_2$, which means that $\M$ is of general type and  $(u,v)$ are canonical isotropic coordinates.
Comparing  \eqref{H_K-Null_R^3_1-tl} with \eqref{LMN-FH-R^3_1-tl} we see that the mean curvature of $\M$ is the given function $H$. 
Hence, the surface  $\M$ has the necessary properties.

 Moreover, if  $\hat\M$ is another surface with the same properties, then \eqref{LMN-FH-R^3_1-tl} is  also valid for $\hat\M$.
So,  $\M$ and $\hat\M$ have one and the same coefficients of the first and second fundamental form. Hence,  according to Theorem \ref{Thm-Bonnet_FLMN_R31-tl}, $\hat\M$ is obtained from  $\M$ by a proper motion in  $\RR^3_1$.
\end{proof}

Now, we will consider the natural equation \eqref{Nat_Eq_R^3_1-tl} in the case of a surface with constant mean curvature $H$. In this case,  equation \eqref{Nat_Eq_R^3_1-tl}  takes the form:
\begin{equation}\label{Nat_Eq_H_cnst_R^3_1-tl}
\frac{F F_{uv} -  F_u F_v}{F} = \varepsilon_1 \varepsilon_2 - F^2H^2,
\end{equation}
and \eqref{LMN-FH-R^3_1-tl} implies $L=\varepsilon_1$, $N=\varepsilon_2$.
Then, it follows from \eqref{H2K-Null_R^3_1-tl} that
\begin{equation}\label{H2K_H_cnst-Null_R^3_1-tl}
H^2-K = \frac{\varepsilon_1 \varepsilon_2}{F^2};  \qquad  F = \frac{1}{\sqrt{|H^2-K|}}.
\end{equation}

If we rewrite \eqref{Nat_Eq_H_cnst_R^3_1-tl} in the form
\begin{equation}\label{Nat_Eq_H_cnst_ln_R^3_1-tl}
\frac{1}{F}(\ln F)_{uv} = \frac{\varepsilon_1 \varepsilon_2}{F^2} - H^2
\end{equation}
then by use of  \eqref{H2K_H_cnst-Null_R^3_1-tl}  we obtain
\begin{equation}\label{Nat_Eq_H_cnst_HK_R^3_1-tl}
\sqrt{|H^2-K|}\big(\ln \sqrt{|H^2-K|}\,\big)_{uv} = K; \qquad H^2-K \neq 0.
\end{equation}
We call \eqref{Nat_Eq_H_cnst_HK_R^3_1-tl} the \textit{natural equation} of constant mean curvature Lorenz surfaces in  $\RR^3_1$. 

So, we can formulate the following Bonnet-type theorem for Lorentz surfaces of constant non-zero mean curvature.

\begin{thm}\label{Thm-Bonnet_Nat_Eq_H_cnst_HK_R31-tl}
Let $\M=(\D ,\x)$ be a Lorentz surface of general type in $\RR^3_1$ with constant non-zero mean curvature $H$ parametrized by canonical isotropic coordinates.  Then, the Gauss curvature $K$ satisfies the natural equation \eqref{Nat_Eq_H_cnst_HK_R^3_1-tl}. If $\hat\M$ is obtained from  $\M$ by a proper motion in  $\RR^3_1$, then $\hat\M$  generates the same solution to 
\eqref{Nat_Eq_H_cnst_HK_R^3_1-tl}. 

Conversely, let $H$ be a non-zero constant and  $K$ be a function of  $(u,v)$ satisfying the natural equation \eqref{Nat_Eq_H_cnst_HK_R^3_1-tl}. 
Then, at least locally, there exist (up to a proper motion in  $\RR^3_1$) exactly two Lorentz surfaces of general type in $\RR^3_1$ 
parametrized by canonical isotropic coordinates, with the constant $H$ as  non-zero mean curvature and the function  $K$ as  Gauss curvature.
\end{thm}

\begin{proof}
We have already seen that the mean  curvature $H$ and the Gauss curvature $K$ of a surface with the given properties satisfy equation \eqref{Nat_Eq_H_cnst_HK_R^3_1-tl}.
Now we will prove the converse.

Given the constant $H$ and the function  $K$ satisfying \eqref{Nat_Eq_H_cnst_HK_R^3_1-tl}, we define a function  
 $F$ and constants $\varepsilon_1=\pm 1$, $\varepsilon_2=\pm 1$ such that 
equalities  \eqref{H2K_H_cnst-Null_R^3_1-tl} hold true.
Then equality  \eqref{Nat_Eq_H_cnst_HK_R^3_1-tl} implies \eqref{Nat_Eq_H_cnst_ln_R^3_1-tl}, which is equivalent to 
 \eqref{Nat_Eq_H_cnst_R^3_1-tl}, the latter being the natural equation 
 \eqref{Nat_Eq_R^3_1-tl} in the case $H$ is constant. Note that the function  $F$ is determined uniquely by  \eqref{H2K_H_cnst-Null_R^3_1-tl}, while  for the choice of $\varepsilon_1$ and $\varepsilon_2$ we have two different options depending on the choice of signs. 
This means that according to Theorem  \ref{Thm-Bonnet_Nat_Eq_FH_R31-tl} we obtain two different Lorentz surfaces $\M_1$ and $\M_2$ parametrized by canonical isotropic coordinates, whose mean curvature is the given constant $H$ and the coefficient of the first fundamental form is the given function $F$.
 Equalities  \eqref{H2K_H_cnst-Null_R^3_1-tl} hold true for both  $\M_1$ and $\M_2$, and hence,  $K$ is determined uniquely by  $F$ and $H$.
Consequently, the Gauss curvature of both $\M_1$ and $\M_2$ is the given function $K$.

The surface  $\M_2$ cannot be obtained from  $\M_1$ by a proper motion in $\RR^3_1$, since the proper motions preserve the signs of 
$\varepsilon_1$ and $\varepsilon_2$, and the signs of these constants are different for $\M_1$ and $\M_2$. 
If $\hat\M$ is another surface with the same properties, then equalities  \eqref{H2K_H_cnst-Null_R^3_1-tl} hold true also for $\hat\M$ .
Hence,  $\M_1$, $\M_2$, and $\hat\M$ have one and the same coefficients of the first fundamental form and equal mean curvatures. 
Moreover, the constants  $\varepsilon_1$ and $\varepsilon_2$ for $\hat\M$ coincide with the constants for one of the two surfaces $\M_1$ or $\M_2$.
So, according to Theorem  \ref{Thm-Bonnet_Nat_Eq_FH_R31-tl}, $\hat\M$ can be obtained from one of the two surfaces $\M_1$ or $\M_2$ by a proper motion in $\RR^3_1$.
\end{proof}

\begin{remark}\label{Rem-Bonnet_Nat_Eq_H_cnst_HK_R31-tl}
The two surfaces obtained in the last theorem are really different.
We have already seen that $\M_2$ cannot be obtained from  $\M_1$ by a proper motion in $\RR^3_1$. Furthermore, 
$\M_2$ cannot be obtained from $\M_1$ by coordinate  change of the form $u = u(\tilde u)$ and  $v = v(\tilde v)$, since such a change preserves the signs of 
 $\varepsilon_1$ and $\varepsilon_2$, according to  \eqref{FLMN_tild_1-R^3_1-tl}.
$\M_2$ cannot be obtained from $\M_1$ also by coordinate  change of the form  $u = u(\tilde v)$ and $v = v(\tilde u)$,
since in such case the sign of  $H$ changes, according to \eqref{H_K-Null_R^3_1-tl} and \eqref{FLMN_tild_2-R^3_1-tl},
but the surfaces  $\M_1$ and $\M_2$ have equal mean curvatures.
Such a pair of surfaces is  presented in Examples \ref{Exmp-Cylnd1_R^3_1-tl} and \ref{Exmp-Cylnd2_R^3_1-tl}.
\end{remark}

 Now we will consider the case of a surface  $\M$ with zero mean curvature $H$, i.e. $\M$  is a minimal surface in $\RR^3_1$.
In this case, equality  \eqref{Nat_Eq_H_cnst_HK_R^3_1-tl} takes the form:
\begin{equation}\label{Nat_Eq_Min_K_R^3_1-tl}
\sqrt{|K|}\big(\ln \sqrt{|K|}\,\big)_{uv} = K; \qquad K \neq 0.
\end{equation}
Let us point out that the Gauss curvature $K$ and the canonical coordinates are invariant under non-proper motions in $\RR^3_1$.
Hence, in this case the surface is determined uniquely by the solution of \eqref{Nat_Eq_Min_K_R^3_1-tl}  up to an arbitrary  motion.
We have the following Bonnet-type theorem for minimal Lorentz surfaces in $\RR^3_1$.

\begin{thm}\label{Thm-Bonnet_Nat_Eq_Min_K_R31-tl}
Let  $\M=(\D ,\x)$ be a minimal Lorentz surface of general type in  $\RR^3_1$ parametrized by canonical isotropic coordinates. 
Then, the Gauss curvature $K$ of $\M$ satisfies the natural equation \eqref{Nat_Eq_Min_K_R^3_1-tl}.
If $\hat\M$ is obtained from   $\M$ by a motion in $\RR^3_1$, then  $\hat\M$  generates the same solution to \eqref{Nat_Eq_H_cnst_HK_R^3_1-tl}. 

Conversely, given a function $K$ of $(u,v)$ satisfying the natural equation  \eqref{Nat_Eq_Min_K_R^3_1-tl}, there exists, at least locally,  a unique (up to a motion in $\RR^3_1$) minimal Lorentz surface of general type parametrized by canonical isotropic coordinates, such that its  Gauss curvature is the given function $K$.  
\end{thm}

\begin{proof}
The proof is similar to the proof of Theorem \ref{Thm-Bonnet_Nat_Eq_H_cnst_HK_R31-tl}, therefore we will not give it in details, we will only point out the difference.
Again, given the  function $K$, we obtain two different minimal surfaces  $\M_1$ and $\M_2$ parametrized by canonical isotropic coordinates, having $K$ as the Gauss curvature and having different signs of 
$\varepsilon_1$ and $\varepsilon_2$.
Let $\tilde\M_1$ be a surface obtained from  $\M_1$ by a non-proper motion in $\RR^3_1$.
Then, the Gauss curvature $K$ and  the signs of $\varepsilon_1$ and $\varepsilon_2$ of  $\tilde\M_1$ are the same as those of  $\M_2$.
So, equalities  \eqref{H2K_H_cnst-Null_R^3_1-tl} imply that $\tilde\M_1$ and $\M_2$ have one and the same  coefficient  $F$ of the first fundamental form. 
Hence, according to Theorem  \ref{Thm-Bonnet_Nat_Eq_FH_R31-tl},  $\M_2$ is obtained from $\tilde\M_1$ by a proper motion in  $\RR^3_1$.
Consequently,  $\M_2$ is obtained from  $\M_1$ by a non-proper motion in $\RR^3_1$.
\end{proof}


\section{Examples}\label{S:Examples_Can_Coord-LorSurf_R31}

In this section, we will consider examples of Lorentz surfaces of general type in $\RR^3_1$, illustrating the developed theory. 
First, we give an example of a minimal Lorentz surface.

\begin{example}\label{Exmp-Ennp1_R^3_1-tl}
Let us consider the surface in $\RR^3_1$, determined by the following parametrization:
 \begin{equation}\label{Ennp1_R^3_1-tl}
\x = \frac{1}{6}(u^3-v^3+3u-3v,\,-u^3+v^3+3u-3v,\,3u^2-3v^2).
\end{equation}
The coefficients of the first and the second fundamental form are:
\begin{equation*}
E=G=0; \quad F=\frac{1}{2}(u-v)^2; \quad L=1; \quad M=0; \quad N=1.
\end{equation*}
The Gauss curvature and the mean curvature of the surface defined above are expressed as follows:
\begin{equation}\label{Ennp1_HK_R^3_1-tl}
K=-\frac{4}{(u-v)^4}; \quad H=0.
\end{equation}

 Hence, the surface defined by \eqref{Ennp1_R^3_1-tl} is a minimal Lorentz surface (of Enneper-type) parametrized by isotropic coordinates. Moreover, the coordinates $(u, v)$ are canonical, since 
$L=N=1$. The Gauss curvature $K$ is negative, so the surface is of first kind according to Definition  \ref{Def-Kind12_R^3_1-tl}. The function $K$ given in \eqref{Ennp1_HK_R^3_1-tl} is  a solution to the  
natural equation \eqref{Nat_Eq_Min_K_R^3_1-tl}.
\end{example}

\begin{example}\label{Exmp-Ennp2_R^3_1-tl}
Let us consider the surface in $\RR^3_1$, defined by
\begin{equation}\label{Ennp2_R^3_1-tl}
\x = \frac{1}{6}(u^3-v^3+3u-3v,\,-u^3+v^3+3u-3v,\,3u^2+3v^2).
\end{equation}
The coefficients of the first and the second fundamental form are:
\begin{equation*}
E=G=0; \quad F=\frac{1}{2}(u+v)^2; \quad L=1; \quad M=0; \quad N=-1.
\end{equation*}
The Gauss curvature and the mean curvature are expressed as follows:
\begin{equation}\label{Ennp2_HK_R^3_1-tl}
K=\frac{4}{(u+v)^4}; \quad H=0.
\end{equation}

 As in the previous example, \eqref{Ennp2_R^3_1-tl} defines a minimal Lorentz surface of Enneper-type parametrized by canonical isotropic coordinates. 
In this example,  the Gauss curvature $K$ is positive  and hence, the surface is of second kind according to  Definition \ref{Def-Kind12_R^3_1-tl}.
The function $K$ given in \eqref{Ennp2_HK_R^3_1-tl} is  also a solution to the  
natural equation \eqref{Nat_Eq_Min_K_R^3_1-tl}.
\end{example}

Now, we will give examples of  surfaces with  non-zero constant mean curvature.

\begin{example}\label{Exmp-Sphere_R^3_1-tl}
We consider the Lorentz sphere  in $\RR^3_1$ parametrized by isothermal coordinates  $(t,s)$ as follows:
\begin{equation*}
\x = (\sinh t \, \sech s,\, \cosh t \, \sech s ,\,\tanh s).
\end{equation*}
Changing the coordinates with isotropic ones, we obtain:
\begin{equation}\label{Sphere_Null_R^3_1-tl}
\x = (\sinh(u-v) \sech(u+v) ,\,\cosh(u-v) \sech(u+v),\,\tanh(u+v)).
\end{equation}
The coefficients of the first and the second fundamental form are:
\begin{equation*}
E=G=0; \quad F=2\sech^2(u+v); \quad L=0; \quad M=2\sech^2(u+v); \quad N=0.
\end{equation*}
The Gauss curvature and the mean curvature are given by:
\begin{equation*}
K=1; \quad H=1.
\end{equation*}

Hence, the surface defined by \eqref{Sphere_Null_R^3_1-tl} is a Lorentz surface with  non-zero constant mean curvature parametrized by isotropic coordinates. 
In this example,  $H^2-K=0$ which means that the surface is not of general type within the meaning of Definition \ref{Def-Gen_Typ_R^3_1-tl}. 
In this case, we cannot introduce canonical coordinates in the sense of Definition  \ref{Def_Can_R31-tl}.
\end{example}

\begin{example}\label{Exmp-Cylnd1_R^3_1-tl}
Let us consider the cylinder in $\RR^3_1$, parametrized by isothermal coordinates $(t,s)$ as follows:
\begin{equation*}
\x = (t,\,\cos s,\,\sin s).
\end{equation*}
Changing the coordinates with isotropic ones, we obtain:
\begin{equation*}
\x = (u-v,\,\cos (u+v),\,\sin (u+v)).
\end{equation*}
The coefficients of the first and the second fundamental form are:
\begin{equation*}
E=G=0; \quad F=2; \quad L=1; \quad M=1; \quad N=1.
\end{equation*}
The Gauss curvature and the mean curvature are given by:
\begin{equation*}
K=0; \quad H=\frac{1}{2}.
\end{equation*}

 This is an example of a Lorentz surface with non-zero constant mean curvature parametrized by canonical isotropic coordinates. 
It corresponds to the trivial (zero) solution to equation  \eqref{Nat_Eq_H_cnst_HK_R^3_1-tl}.
\end{example}

\begin{example}\label{Exmp-Cylnd2_R^3_1-tl}
Now, let us consider the hyperbolic Lorentz cylinder in $\RR^3_1$, parametrized by isothermal coordinates $(t,s)$ as follows:
\begin{equation*}
\x = (\sinh s,\,\cosh s,\, t).
\end{equation*}
Changing the coordinates with isotropic ones, we obtain:
\begin{equation*}
\x = (\sinh (u-v),\,\cosh (u-v),\, u+v).
\end{equation*}
The coefficients of the first and the second fundamental form are:
\begin{equation*}
E=G=0; \quad F=2; \quad L=-1; \quad M=1; \quad N=-1.
\end{equation*}
The Gauss curvature and the mean curvature are given by:
\begin{equation*}
K=0; \quad H=\frac{1}{2}.
\end{equation*}

 This is also an example of a Lorentz surface with non-zero constant mean curvature parametrized by canonical isotropic coordinates. 
It also corresponds to the trivial (zero) solution to equation  \eqref{Nat_Eq_H_cnst_HK_R^3_1-tl}.
\end{example}

 Comparing the results of the last two examples, we see
that the two cylinders have equal constant mean curvatures and equal Gauss  curvatures. Hence, they give one and the same solution to the natural equation 
 \eqref{Nat_Eq_H_cnst_HK_R^3_1-tl}.
But there is a difference in the signs of  $\varepsilon_1=L$ and $\varepsilon_2=N$. These two cylinders form a pair of surfaces 
 $\M_1$ and $\M_2$ as the ones described in the proof of Theorem  \ref{Thm-Bonnet_Nat_Eq_H_cnst_HK_R31-tl}.

\medskip

Finally, we will consider a surface with non-constant  mean curvature.

\begin{example}\label{Exmp-Cone1_R^3_1-tl}
Let us consider the hyperbolic Lorentz cone  $\M$ in $\RR^3_1$, parametrized by isothermal coordinates $(t,s)$ as follows:
\begin{equation*}
\x = \big(\ee^\frac{t}{2} \sinh s,\, \sqrt{3}\,\ee^\frac{t}{2},\, \ee^\frac{t}{2} \cosh s\big).
\end{equation*}
Changing the coordinates with isotropic ones, we obtain:
\begin{equation*}
\x = \big(\ee^\frac{u+v}{2} \sinh (u-v),\, \sqrt{3}\,\ee^\frac{u+v}{2},\, \ee^\frac{u+v}{2} \cosh (u-v)\big).
\end{equation*}
The coefficients of the first and the second fundamental form are:
\begin{equation*}
E=G=0; \quad F=2\ee^{u+v}; \quad 
L=\frac{\sqrt{3}}{2}\ee^\frac{u+v}{2}; \quad M=-\frac{\sqrt{3}}{2}\ee^\frac{u+v}{2}; \quad N=\frac{\sqrt{3}}{2}\ee^\frac{u+v}{2}.
\end{equation*}
The Gauss curvature and the mean curvature are given by:
\begin{equation*}
K=0; \quad H=-\frac{\sqrt{3}}{4}\ee^{-\frac{u+v}{2}}.
\end{equation*}

 Hence, in this example,  $\M$ is a Lorentz surface with non-constant  mean curvature parametrized by isotropic coordinates. The coordinates $(u,v)$ are not canonical, since 
 $L(u,v_0) \neq \pm1$ and  $N(u_0,v) \neq \pm1$. 

 We will introduce canonical  isotropic coordinates  $(\tilde u,\tilde v)$ with initial point $(u_0,v_0)=(0,0)$.
Using formulas  \eqref{eq_can-sol_R31-tl} we get:
\begin{equation*}
\tilde u = \tilde u_0 + 2\sqrt{2}\sqrt[4]{3}(\ee^\frac{u}{4} - 1); \qquad \tilde v = \tilde v_0 + 2\sqrt{2}\sqrt[4]{3}(\ee^\frac{v}{4} - 1).
\end{equation*}
To simplify the formulas we choose $\tilde u_0 = \tilde v_0 = 2\sqrt{2}\sqrt[4]{3}$. Then,  
\begin{equation*}
\tilde u = 2\sqrt{2}\sqrt[4]{3}\,\ee^\frac{u}{4}; \qquad \tilde v = 2\sqrt{2}\sqrt[4]{3}\,\ee^\frac{v}{4}; \qquad
u = 4\ln\frac{\tilde u}{2\sqrt{2}\sqrt[4]{3}}; \qquad v = 4\ln\frac{\tilde v}{2\sqrt{2}\sqrt[4]{3}}.
\end{equation*}

Using the last equalities and formulas  \eqref{FLMN_tild_1-R^3_1-tl}, we express the coefficient of the first fundamental form   $\tilde F$ and the mean curvature  $\tilde H$  in terms of the canonical coordinates  $(\tilde u,\tilde v)$ as follows:
\begin{equation}\label{Cone1_tldFH_R^3_1-tl}
\tilde F = \frac{\tilde u^3 \tilde v^3}{1152}; \qquad
\tilde H =-\frac{48\sqrt{3}}{\tilde u^2 \tilde v^2}.
\end{equation}
According to Theorem  \ref{Thm-Bonnet_Nat_Eq_FH_R31-tl}, the functions  $\tilde F$ and $\tilde H$ given by \eqref{Cone1_tldFH_R^3_1-tl} give a solution to the natural equation \eqref{Nat_Eq_R^3_1-tl} in the case $\varepsilon_1=\varepsilon_2=1$.
\end{example}

\vskip 5mm \textbf{Acknowledgments:}
The  third author is partially supported by the National Science Fund, Ministry of Education and Science of Bulgaria under contract DN 12/2.

\vskip 3mm


\bibliographystyle{plainnat}

\bibliography{MyBibliography}

\end{document}